\documentclass[a4paper,11pt]{article}
\usepackage[T1]{fontenc}
\usepackage{amsfonts}
\usepackage{textcomp}
\usepackage{amsmath, amsthm, amssymb, mathrsfs}
\usepackage{graphicx, url}

\title{On the Transition Laws of $p$-Tempered $\alpha$-Stable OU-Processes}
\author{Michael Grabchak\footnote{Email address: mgrabcha@uncc.edu}\ \ 
{\it University of North Carolina Charlotte}}

\begin{document}
\newtheorem{prop}{Proposition}
\newtheorem{thrm}{Theorem}
\newtheorem{defn}{Definition}
\newtheorem{cor}{Corollary}
\newtheorem{lemma}{Lemma}
\newtheorem{remark}{Remark}
\newtheorem{exam}{Example}

\newcommand{\rd}{\mathrm d}
\newcommand{\rE}{\mathrm E}
\newcommand{\ts}{\mathrm{TS}^p_\alpha}
\newcommand{\ID}{\mathrm{ID}}
\newcommand{\Ga}{\mathrm{Ga}}
\newcommand{\GGa}{\mathrm{GGa}}
\newcommand{\IGa}{\mathrm{IGa}}
\newcommand{\LL}{\mathrm{LL}}
\newcommand{\tr}{\mathrm{tr}}
\newcommand{\iid}{\stackrel{\mathrm{iid}}{\sim}}
\newcommand{\eqd}{\stackrel{d}{=}}
\newcommand{\cond}{\stackrel{d}{\rightarrow}}
\newcommand{\conv}{\stackrel{v}{\rightarrow}}
\newcommand{\conw}{\stackrel{w}{\rightarrow}}
\newcommand{\conp}{\stackrel{p}{\rightarrow}}
\newcommand{\simp}{\stackrel{p}{\sim}}
\newcommand{\plim}{\mathop{\mathrm{p\mbox{-}lim}}}
\newcommand{\lgg}{\mathrm{log}}
\newcommand{\dlim}{\operatorname*{d-lim}}

\maketitle

\begin{abstract}
We derive an explicit representation for the transition law of a $p$-tempered $\alpha$-stable process of Ornstein-Uhlenbeck-type and use it to develop a methodology for simulation. Our results apply in both the univariate and multivariate cases. Special attention is given to the case where $p\le\alpha$, which is more complicated and requires additional care.\\

\end{abstract}

\section{Introduction}

Tempered stable distributions are a rich and flexible class of models that are obtained by modifying the tails of infinite variance stable distributions to make them lighter. This leads to distributions that are more realistic for a variety of application areas. We are particularly interested in the class of $p$-tempered $\alpha$-stable distributions with $p>0$ and $\alpha\in[0,2)$. This class was introduced in \cite{Grabchak:2012} and was further studied in \cite{Grabchak:2016book}. It contains most of the best known and most heavily used families of tempered stable distributions including the models studied in \cite{Rosinski:2007} and \cite{Bianchi:Rachev:Kim:Fabozzi:2011}, which, themselves, contain important subclasses such as gamma distributions, inverse Gaussian distributions, classical tempered stable distributions (CTS), and rapidly decreasing tempered stable distributions (RDTS).

Associated with each $p$-tempered $\alpha$-stable distribution is a non-Gaussian process of Ornstein-Uhlenbeck-type (henceforth TSOU-process). These processes are mean reverting and are useful for a variety of applications. We are particularly motivated by applications to mathematical finance, where such processes have been used to model stochastic volatility,  stochastic interest rates, and commodity prices, see \cite{Barndorff-Nielsen:Shephard:2001} and the references in \cite{Kawai:Masuda:2012}, \cite{Bianchi:Rachev:Fabozzi:2017}, and \cite{Grabchak:2020}. In this paper, we derive an explicit representation for the transition law of a TSOU-process. We then use this representation to develop a methodology for simulating increments from the process. Our results apply in both the univariate and multivariate cases. Further, while they hold for all values of $\alpha$, we are particularly concerned with the case where $p\le\alpha$, as it requires additional care. In the important case when $p=1$, this corresponds to the case of infinite variation.

For $\alpha<p$, in the special cases of gamma, inverse Gaussian, 
CTS, and RDTS distributions, similar results are given in \cite{Qu:Dassios:Zhao:2019}, \cite{Zhang:Zhang:2008}, 
\cite{Zhang:Zhang:2009}, and \cite{Bianchi:Rachev:Fabozzi:2017}. For details, see Remark \ref{remark: cts} below. The general case with $0<\alpha<p$ was considered in \cite{Grabchak:2020}. However, even for that case, our results often provide a simpler methodology for simulation. To the best of our knowledge, the case $\alpha\ge p$ has only been considered for CTS distributions, see \cite{Kawai:Masuda:2012}.

The rest of this paper is organized as follows. In Section \ref{sec: TS}, we recall the definition of $p$-tempered $\alpha$-stable distributions and give some properties. Then, in Section \ref{sec: IGa}, we introduce the incomplete gamma distribution, which is important for characterizing the transition laws of TSOU-processes. In Section \ref{sec: TSOU}, we formally define TSOU-processes and characterize their transition laws. In Section \ref{sec: sim TS}, we discuss how to use these results for simulation. A small-scale simulation study is given in Section \ref{sec: sim study}.  Proofs are postponed to Section \ref{sec: proofs}.

Before proceeding, we introduce some notation. Let $\mathbb R^d$ be the space of $d$-dimensional column vectors of real numbers equipped with the usual inner product $\langle\cdot,\cdot\rangle$ and the usual norm $|\cdot|$. Let $\mathbb S^{d-1}=\{x\in\mathbb R^d: |x|=1\}$ denote the unit sphere in $\mathbb R^d$. Let $\mathfrak B(\mathbb R^d)$ and $\mathfrak B(\mathbb S^{d-1})$ denote the Borel sets in $\mathbb R^d$ and $\mathbb S^{d-1}$, respectively. For a Borel measure $M$ on $\mathbb R^d$ and $s\ge0$, we write $sM$ to denote the Borel measure on $\mathbb R^d$ given by $(sM)(B)=sM(B)$ for $B\in\mathfrak B(\mathbb R^d)$. If $a,b\in\mathbb R$, we write $a\vee b$ and $a\wedge b$ to denote, respectively, the maximum and the minimum of $a$ and $b$. If $\mu$ is a probability measure on $\mathbb R^d$, we write $X\sim\mu$ to denote that $X$ is an $\mathbb R^d$-valued random variable with distribution $\mu$ and we write $X_1,X_2,\dots\iid\mu$ to denote that $X_1,X_2,\dots$ are independent and identically distributed $\mathbb R^d$-valued random variables each with distribution $\mu$. For two random variables $X$ and $Y$, we write $X \eqd Y$ to denote that $X$ and $Y$ have the same distribution. We write $U(a,b)$ to denote the uniform distribution on the interval $(a,b)$, $\mathrm{Pois}(\psi)$ to denote the Poisson distribution with a mean of $\psi$, and $\delta_x$ to denote the point mass at $x$. We write $1_A$ to denote the indicator function on set $A$. For sums, we interpret $\sum_{n=1}^0$ as $0$. If $f$ and $g$ are positive functions and $a\in[0,\infty]$, we write $f(x)\sim g(x)$ as $x\to a$ to denote $\lim_{x\to a}f(x)/g(x)\to 1$. For $x\ge0$, we write $\lfloor x\rfloor$ to denote the integer part of $x$.

\section{Tempered Stable Distributions}\label{sec: TS}

An infinitely divisible distribution $\mu$ on $\mathbb R^d$ is a probability measure with a characteristic function of the form $\hat\mu(z) = \exp\{C_\mu(z)\}$, where, for $z\in\mathbb R^d$,
\begin{eqnarray*}
C_\mu(z) = -\frac{1}{2}\langle z,Az\rangle + i\langle b,z\rangle + \int_{\mathbb R^d}\left(e^{i\langle z,x\rangle}-1-i\langle z,x\rangle h(x)\right)M(\rd x).
\end{eqnarray*}
Here, $A$ is a symmetric nonnegative-definite $d\times d$-dimensional matrix called the Gaussian part, $b\in\mathbb R^d$ is called the shift, and $M$ is a Borel measure, called the L\'evy measure, which satisfies
\begin{eqnarray}\label{eq: cond for levy measure}
M(\{0\})=0 \mbox{\ and\ } \int_{\mathbb R^d}(|x|^2\wedge1)M(\rd x)<\infty.
\end{eqnarray}
The function $h:\mathbb R^d\mapsto\mathbb R$, which we call the $h$-function, can be any Borel function satisfying
\begin{eqnarray*}
\int_{\mathbb R^d}\left|e^{i\langle z,x\rangle}-1-i\langle z,x\rangle h(x)\right| M(\rd x)<\infty\ \mbox{  for all  } z\in\mathbb R^d.
\end{eqnarray*}
For a fixed $h$-function, the parameters $A$, $M$, and $b$ uniquely determine the distribution $\mu$, and we write $\mu=\ID(A,M,b)_h$. The choice of $h$ does not affect $A$ and $M$, but different choices of $h$ require different values for $b$, see Section 8 in \cite{Sato:1999}. 

Associated with every infinitely divisible distribution $\mu=\ID(A,M,b)_h$ is a L\'evy process, $\{X_t:t\ge0\}$, which is stochastically continuous and has independent and stationary increments. The characteristic function of $X_t$ is $\left(\hat\mu(z)\right)^t$. It follows that, for each $t\ge0$, $X_t\sim\ID(tA,tM,tb)_h$. For more on infinitely divisible distributions and their associated L\'evy processes see \cite{Sato:1999}.

A $p$-tempered $\alpha$-stable distribution on $\mathbb R^d$ is an infinitely divisible distribution with no Gaussian part and a L\'evy measure of the form
\begin{eqnarray}\label{eq: levy ts}
L(B) = \int_{\mathbb S^{d-1}}\int_0^\infty 1_{B}(u\xi) u^{-1-\alpha} q(\xi,u^p)\rd u \sigma(\rd \xi), \ \ B\in\mathfrak B(\mathbb R^d),
\end{eqnarray}
where $p>0$, $\alpha\in[0,2)$, $\sigma$ is a finite Borel measure on $\mathbb S^{d-1}$, 
\begin{eqnarray}\label{eq: q to p temp}
q(\xi,u) = \int_{(0,\infty)} e^{-su} Q_\xi(\rd s),
\end{eqnarray}
and $\bar Q=\{Q_\xi:\xi\in\mathbb S^{d-1}\}$ is a measurable family of probability measures on $(0,\infty)$. When $\alpha=0$ we need the additional assumption that
$$
\int_1^\infty q(\xi,u)u^{-1}\rd u <\infty \ \mbox{for } \sigma\mbox{-a.e.}\ \xi
$$
to ensure that $L$ satisfies \eqref{eq: cond for levy measure}. The class of $p$-tempered $\alpha$-stable distributions was introduced in \cite{Grabchak:2012} and was further studied in the monograph \cite{Grabchak:2016book}. The case where $p=1$ had previously been introduced in \cite{Rosinski:2007} and the case where $p=2$ had previously been introduced in \cite{Bianchi:Rachev:Kim:Fabozzi:2011}. The case where $p=1$ and $\alpha=0$ corresponds to a large subclass of Thorin's class of generalized gamma convolutions, see \cite{Barndorff-Nielsen:Maejima:Sato:2006} and the references therein.

It is often convenient to work with a different representation of the L\'evy measure. Toward this end, define the Borel measures
\begin{eqnarray*}
Q(B) = \int_{\mathbb S^{d-1}}\int_{(0,\infty)} 1_{B}(s\xi) Q_\xi(\rd s)\sigma(\rd \xi), \ \ B\in\mathfrak B(\mathbb R^d)
\end{eqnarray*}
and
\begin{eqnarray}\label{eq: R}
R(B) = \int_{\mathbb R^{d}} 1_{B}\left(\frac{x}{|x|^{1+1/p}}\right) |x|^{\alpha/p}Q(\rd x), \ \ B\in\mathfrak B(\mathbb R^d).
\end{eqnarray}
From $R$ we can recover $Q$ by
\begin{eqnarray*}
Q(B) = \int_{\mathbb R^{d}}1_{B}\left( \frac{x}{|x|^{1+p}}\right)|x|^{\alpha} R(\rd x), \ \ B\in\mathfrak B(\mathbb R^d).
\end{eqnarray*}
It can be shown that the L\'evy measure, as given by \eqref{eq: levy ts}, can be written as
\begin{eqnarray}\label{eq: L pTS}
L(B) = \int_{\mathbb R^{d}}\int_0^\infty 1_{B}(u x) u^{-1-\alpha} e^{-u^p} \rd u R(\rd x), \ \ B\in\mathfrak B(\mathbb R^d)
\end{eqnarray}
and the measure $\sigma$ can be written as
\begin{eqnarray*}
\sigma(B) = \int_{\mathbb R^d}1_B\left(\frac{x}{|x|}\right)|x|^\alpha R(\rd x),\ \ B\in\mathfrak B(\mathbb S^{d-1}),
\end{eqnarray*}
see Chapter 3 in \cite{Grabchak:2016book}. Further, for fixed $\alpha\in[0,2)$ and $p>0$, the measure $R$ uniquely determines the L\'evy measure $L$. The measure $R$ is called the Rosi\'nski measure of the distribution, after the author of \cite{Rosinski:2007}. A Borel measure $R$ on $\mathbb R^d$ is the Rosi\'nski measure of some $p$-tempered $\alpha$-stable distribution if and only if $R(\{0\})=0$ and 
\begin{eqnarray}\label{eq: cond for R measure}
\begin{array}{ll}\int_{\mathbb R^d}|x|^\alpha R(\rd x)<\infty & \mbox{if } \alpha\in(0,2)\\
\int_{|x|\le2}R(\rd x)+\int_{|x|>2}\log|x|R(\rd x)<\infty & \mbox{if } \alpha=0
\end{array}.
\end{eqnarray}
For simplicity, when $\alpha=1$, we generally make the slightly stronger assumption that
\begin{eqnarray}\label{eq: alpha 1 cond}
 \int_{|x|\le2}|x|R(\rd x) +\int_{|x|>2}|x|\log|x| R(\rd x)<\infty,
\end{eqnarray}
which guarantees that the corresponding distribution has a finite mean. When $R$ satisfies \eqref{eq: cond for R measure} (and, if $\alpha=1$, \eqref{eq: alpha 1 cond}), we can use the $h$-function
\begin{eqnarray*}
h_\alpha(x) = 1_{[\alpha\ge1]}= \left\{\begin{array}{ll}
0& \mbox{if }\alpha\in[0,1)\\
1 & \mbox{if } \alpha\in[1,2)
\end{array}\right..
\end{eqnarray*}

\begin{defn}
Fix $\alpha\in[0,2)$ and $p>0$. Let $R$ be a Borel measure on $\mathbb R^d$ with $R(\{0\})=0$ such that \eqref{eq: cond for R measure} holds. If $\alpha=1$, assume further that \eqref{eq: alpha 1 cond} holds. We write $\mathrm{TS}^p_\alpha(R,b)$ to denote the distribution $\ID(0,L,b)_{h_\alpha}$, where $L$ is of the form \eqref{eq: L pTS} and $b\in\mathbb R^d$.
\end{defn}

We note that one can define a more general class of distributions with L\'evy measures of the form \eqref{eq: L pTS}, but where $R$ does not satisfy \eqref{eq: cond for R measure}. In \cite{Grabchak:2016book}, distributions where $R$ satisfies \eqref{eq: cond for R measure} are called ``proper $p$-tempered $\alpha$-stable distributions.''

\section{Incomplete Gamma Distribution}\label{sec: IGa}

In this section we introduce the incomplete gamma distribution, which is important for studying the transition laws of TSOU-processes and is needed in Theorem \ref{thrm: main} below. It may also be of independent interest. We begin by recalling that the probability density function (pdf) of a gamma distribution is of the form
$$
\frac{\zeta^\gamma}{\Gamma(\gamma)}u^{\gamma-1}e^{-u \zeta}, \ \ u>0,
$$
where $\gamma,\zeta>0$ are parameters. We denote this distribution by $\Ga(\gamma,\zeta)$. Let 
\begin{eqnarray*}
G_{\gamma,\zeta}(u) &=& \frac{\zeta^\gamma}{\Gamma(\gamma)}\int_0^u x^{\gamma-1}e^{-x\zeta}\rd x\\
&=& \frac{1}{\Gamma(\gamma)}\int_0^{u\zeta} x^{\gamma-1}e^{-x}\rd x, \ \ u>0
\end{eqnarray*}
be the cumulative distribution function (cdf) of this distribution. When $\gamma$ is a positive integer, then Lemma \ref{lemma: exp approx} in Section \ref{sec: proofs} below gives
\begin{eqnarray}\label{eq:gamma cdf sum}
G_{\gamma,\zeta}(u) = 1-e^{-u\zeta}\sum_{n=0}^{\gamma-1}\frac{\zeta^n u^n}{n!}, \ \ u>0.
\end{eqnarray}
Consider the new pdf defined by
\begin{eqnarray*}
f_{\beta,\gamma,p,\eta}(u) &=& \frac{1}{K_{\beta,\gamma,p,\eta}}  G_{\gamma,(\eta-1)}\left(u^p\right) e^{-u^p} u^{-1-\beta}, \ \ u>0
\end{eqnarray*}
where $K_{\beta,\gamma,p,\eta}>0$ is a normalizing constant and $\beta\in\mathbb R$, $\gamma>0$, $p>0$, $\eta>1$ are parameters satisfying  $p\gamma>\beta$. Since $G_{\gamma,(\eta-1)}$ is, essentially, an incomplete gamma function, we refer to the distribution with pdf $f_{\beta,\gamma,p,\eta}$ as the incomplete gamma distribution and denote it by $\IGa(\beta,\gamma,p,\eta)$. When $\gamma$ is a positive integer, \eqref{eq:gamma cdf sum} implies that 
\begin{eqnarray*}
f_{\beta,\gamma,p,\eta}(u)  &=&  \frac{1}{K_{\beta,\gamma,p,\eta}} \left(e^{-u^p} - e^{-u^p \eta} \sum_{n=0}^{\gamma-1}  \frac{(\eta-1)^n }{n!}u^{np} \right)  u^{-1-\beta}, \ \ u>0.
\end{eqnarray*}
Note that if $X\sim \IGa(\beta,\gamma,p,\eta)$ and $\kappa>\beta-p\gamma$, then
$$
\rE[X^\kappa] = \frac{K_{(\beta-\kappa),\gamma,p,\eta}}{K_{\beta,\gamma,p,\eta}}.
$$
We now give some facts about $K_{\beta,\gamma,p,\eta}$.

\begin{prop}\label{prop: K}
We have
\begin{eqnarray*}
K_{\beta,\gamma,p,\eta} = \frac{\Gamma(\gamma-\beta/p)}{p\Gamma(\gamma)}\int^{1}_{1/\eta}(1-u)^{\gamma-1}u^{-1-\beta/p} \rd u,
\end{eqnarray*}
$$
K_{\beta,\gamma,p,\eta}  \sim \frac{\Gamma(\gamma-\beta/p)}{p\Gamma(\gamma+1)} (\eta-1)^\gamma \mbox{ as } \eta\downarrow1.
$$
and, as $\eta\to\infty$
$$
K_{\beta,\gamma,p,\eta}  \sim \left\{\begin{array}{ll}\frac{\Gamma(\gamma-\beta/p)}{\beta\Gamma(\gamma)} \eta^{\beta/p}, & \beta>0\\
p^{-1} \log(\eta), & \beta=0\\
p^{-1}\Gamma(|\beta|/p), & \beta<0
\end{array}\right..
$$
Further, if $\gamma$ is a positive integer, then
\begin{eqnarray*}
K_{\beta,\gamma,p,\eta} = \frac{\Gamma(\gamma-\beta/p)}{(\gamma-1)!} \sum_{n=0}^{\gamma-1} \binom{\gamma-1}{n} (-1)^n \frac{1-\eta^{-(np-\beta)/p}}{np-\beta},
\end{eqnarray*}
where in the case $np=\beta$ we interpret $\frac{1-\eta^{-(np-\beta)/p}}{np-\beta}$ by its limiting value of $p^{-1}\ln \eta$.
\end{prop}

We now develop an accept-reject algorithm to simulate from $\IGa(\beta,\gamma,p,\eta)$. Toward this end, recall that the generalized gamma distribution has a pdf of the form
$$
\frac{p\zeta^{\gamma/p}}{\Gamma(\gamma/p)}u^{\gamma-1}e^{-u^p \zeta},\ u>0
$$
where $\gamma,p,\zeta>0$ are parameters, see  \cite{Stacy:1962}. We denote this distribution by $\GGa(\gamma,p,\zeta)$. It is readily checked that 
\begin{eqnarray}\label{eq: sim GGa}
\mbox{if }X\sim \Ga(\gamma/p,\zeta)\mbox{, then }X^{1/p}\sim \GGa(\gamma,p,\zeta).
\end{eqnarray}

\begin{prop}\label{prop: IGa bound}
We have
$$
f_{\beta,\gamma,p,\eta}(u) \le V_1 g_1(u), \ \ u>0,
$$
where $g_1$ is the pdf of the $\GGa(p\gamma-\beta,p,1)$ distribution and
\begin{eqnarray*}
V_1 &=&\frac{(\eta-1)^\gamma}{p K_{\beta,\gamma,p,\eta}}\frac{\Gamma(\gamma-\beta/p)}{\Gamma(\gamma+1)}.
\end{eqnarray*}
\end{prop}

Let $\varphi_1(u) =  f_{\beta,\gamma,p,\eta}(u^{1/p})/(V_1 g_1(u^{1/p})))$ and note that
\begin{eqnarray*}
\varphi_1(u) &=& \frac{\Gamma(\gamma+1)}{(\eta-1)^\gamma} G_{\gamma,(\eta-1)}(u) u^{-\gamma}.
\end{eqnarray*}
With this notation and taking \eqref{eq: sim GGa} into account, we get the following accept-reject algorithm for simulating from $\IGa(\beta,\gamma,p,\eta)$.\\

\noindent \textbf{Algorithm 1.}\\
\textbf{Step 1.} Independently simulate $U\sim U(0,1)$ and $Y\sim \Ga(\gamma-\beta/p,1)$.\\
\textbf{Step 2.} If $U\le \varphi_1(Y)$ return $Y^{1/p}$, otherwise go back to step 1.\\

On a given iteration, the probability of acceptance is $1/V_1$. Note that, by Proposition \ref{prop: K}, $1/V_1\to1$ as $\eta\downarrow 1$ and $1/V_1\to0$ as $\eta\to\infty$. Thus, this algorithm tends to work better when $\eta$ is close to $1$. As we will see, this is the regime that we are most interested in.

\section{TSOU-Processes}\label{sec: TSOU}

In this section we formally define TSOU-processes and characterize their transition laws. We begin by recalling the definition of a process of Ornstein-Uhlenbeck-type (henceforth OU-process). Let $Z=\{Z_t:t\ge0\}$ be a L\'evy process with $Z_1\sim \ID(A,M,b)_h$ and define a process $Y=\{Y_t:t\ge0\}$ by the stochastic differential equation
$$
\rd Y_t = -\lambda Y_t\rd t + \rd Z_t,
$$
where $\lambda>0$ is a parameter. This has a strong solution of the form
\begin{eqnarray*}
Y_t = e^{-\lambda t}Y_0 + \int_0^t e^{-\lambda(t-s)} \rd Z_s.
\end{eqnarray*}
In this case $Y$ is called an OU-process with parameter $\lambda$ and $Z$ is called the background driving L\'evy process (BDLP). The process $Y$ is a Markov process and so long as
\begin{eqnarray*}
\int_{|x|>2} \log|x| M(\rd x)<\infty
\end{eqnarray*}
it has a limiting distribution. This distribution is necessarily selfdecomposable. Further, every selfdecomposable distribution is the limiting distribution of some OU-process. For details see \cite{Sato:1999} or \cite{Rocha-Arteaga:Sato:2019}.

Theorem 15.10 in \cite{Sato:1999} implies that all $p$-tempered $\alpha$-stable distributions are selfdecomposable and, thus, that each is the limiting distribution of some OU-process. We refer to these as $p$-tempered $\alpha$-stable OU-processes or TSOU-processes. We now characterize the BDLP of a TSOU-process.

\begin{prop}\label{prop: BDLP for TS}
The BDLP of a TSOU-process with parameter $\lambda>0$ and limiting distribution $\ts(R,b)$ is the L\'evy process $\{Z_t:t\ge0\}$ with $Z_1\sim \ID(0,\lambda M, \lambda b)_{h_\alpha}$, where for $B\in\mathfrak B(\mathbb R^d)$
\begin{eqnarray*}
M(B) &=& \int_{\mathbb R^d} \int_0^\infty 1_{B}\left(u x\right)\left(\alpha+p u^{p} \right)  u^{-1-\alpha} e^{-u^{p}} \rd u R(\rd x).
\end{eqnarray*}
\end{prop}

For $p=1$ this is given in \cite{Rosinski:2007}.  Some related results are given in \cite{Terdik:Woyczynski:2006}. We now give our main result, which is an explicit representation for the transition function of a TSOU-process that can be used for simulation.

\begin{thrm}\label{thrm: main}
Let $Y=\{Y_t:t\ge0\}$ be a TSOU-process with parameter $\lambda>0$ and limiting distribution $\ts(R,b)$ with $p>0$, $\alpha\in[0,2)$ and $R$ satisfying \eqref{eq: cond for R measure} (or if $\alpha=1$, \eqref{eq: alpha 1 cond}). Assume, in addition, that $0<R(\mathbb R^d)<\infty$ and set $\gamma=1+\lfloor\alpha/p\rfloor$. If $t>0$, then, given $Y_s=y$, we have
\begin{eqnarray}\label{eq: trans RV}
Y_{s+t} \eqd e^{-\lambda t} y + (1-e^{-\lambda t})b - \sum_{n=0}^{\gamma-1}b_n + X_0 + e^{-\lambda t} \sum_{n=1}^{\gamma-1}X_n + \sum_{j=1}^{N} V_j W_j, 
\end{eqnarray}
where $b_0,\dots,b_{\gamma-1}\in\mathbb R^d$ are constants and $N, X_0, X_1,\dots,X_{\gamma-1}$, $V_1, V_2, \dots$, $W_1, W_2, \dots$ are independent random variables with:\\
1. $X_0\sim \ts(R_0,0)$ with $R_0(\rd x) = (1-e^{-\alpha \lambda t})R(\rd x)$,\\
2. if $\gamma\ge2$ then $X_n \sim \mathrm{TS}^p_{\alpha-np}(R_n,0)$ with $R_n(\rd x) = \frac{1}{n!}(1-e^{-p \lambda t})^n R(\rd x)$ for $n=1,2,\dots,(\gamma-1)$,\\
3. $V_1, V_2, \dots\iid R^1$, where $R^1(\rd x)=R(\rd x)/R(\mathbb R^{d})$,\\
4. $W_1, W_2, \dots\iid\IGa(\alpha,\gamma,p,e^{p\lambda t})$,\\
5. $N$ has a Poisson distribution with mean $e^{-\alpha \lambda t}R(\mathbb R^d)K_{\alpha,\gamma,p,e^{p\lambda t}}$,\\
6.
$$
b_0 = \left\{\begin{array}{ll} e^{-\alpha\lambda t} \int_{\mathbb R^{d}} x  R(\rd x) K_{(\alpha-1),\gamma,p,e^{p\lambda t}} 
& \alpha\in[1,2)\\
0 & \alpha\in[0,1)
\end{array}\right.,
$$
and if $\gamma\ge2$ then for $n=1,2,\dots,(\gamma-1)$
$$
b_n =  \left\{\begin{array}{ll} 
e^{-\lambda t} \int_{\mathbb R^{d}}x R_n(\rd x)  p^{-1}\Gamma\left(\frac{1-\alpha+np}{p}\right)   & 1\le \alpha<1+ np\\
0 & \mbox{otherwise}
\end{array}\right..
$$
\end{thrm}

We note that, in the case $0<\alpha<p$, a version of this result is contained in Theorem 2 of \cite{Grabchak:2020}. However, in that paper, the distribution of the product $V_jW_j$ is presented in a less intuitive way.  In situations where it is easy to simulate from $R^1$, the representation given in Theorem \ref{thrm: main} leads to a methodology for simulation that is simpler than the one suggested by the results in \cite{Grabchak:2020}.

\begin{remark}\label{remark: cts}
CTS distributions are one dimension distributions of the form $\ts(R,b)$, where  $p=1$ and $R(\rd x)= a\eta^{\alpha}\delta_{1/\eta}(\rd x)$ for some $a,\eta>0$. When $\alpha=0$, which corresponds to the class of gamma distributions, a version of Theorem \ref{thrm: main} can be found in \cite{Qu:Dassios:Zhao:2019} and when $\alpha=.5$, which  corresponds to the class of inverse Gaussian distributions, it can be found in \cite{Zhang:Zhang:2008}. More generally, for $\alpha\in(0,1)$ it can be found in \cite{Zhang:Zhang:2009}, and for $\alpha\in(1,2)$ it can be found in \cite{Kawai:Masuda:2012}. RDTS distributions are extensions of CTS distributions to the case where $p=2$. In this case the result can be found in \cite{Bianchi:Rachev:Fabozzi:2017}. The result in \cite{Kawai:Masuda:2012} is the only case with $\alpha\ge p$ that we have seen in the literature.
\end{remark}

\begin{remark}\label{remark: alpha 1}
A version of Theorem \ref{thrm: main} can be obtained when $\alpha=1$ and \eqref{eq: cond for R measure} holds, but \eqref{eq: alpha 1 cond} does not. In this case we use the $h$-function given by $h(x)=1_{[|x|\le1]}$. Let $R$ be a Rosi\'nski measure with $0<\int_{\mathbb R^d}\left(|x|\vee1\right)R(\rd x)<\infty$ and let $L$ be the L\'evy measure given by \eqref{eq: L pTS} with $\alpha=1$ and some $p>0$. Arguments similar to those in the proof of Proposition \ref{prop: BDLP for TS} imply that a TSOU-process with parameter $\lambda>0$ and limiting distribution $\ID(0,L,b)_h$ has BDLP $Z=\{Z_t:t\ge0\}$ with $Z_1\sim \ID(0,\lambda M, \lambda c)_h$, where $M$ is as in Proposition \ref{prop: BDLP for TS} and 
$$
c = b - \int_{\mathbb R^d}\frac{x}{|x|} \int_{|x|^{-1}}^\infty (1+pu^p)u^{-2}e^{-u^p}\rd u R(\rd x) = b - \int_{\mathbb R^d} x e^{-|x|^{-p}} R(\rd x),
$$
where the last equality follows from \eqref{eq: deriv} below. In this case, \eqref{eq: trans RV} holds, but with $X_0\sim \ID(0,(1-e^{-\lambda t})L,0)_h$, $b_n=0$ for $n>0$, and
\begin{eqnarray*}
b_0  &=& \left(1-e^{-\lambda t}\right) \int_{\mathbb R^d} x e^{-|x|^{-p}} R(\rd x)\\
&&\quad + \int_{\mathbb R^d} x \int_{|x|^{-1}}^{|x|^{-1}e^{\lambda t}}\left(\frac{1}{|x|u}-e^{-\lambda t}\right)(1+pu^p) u^{-1}e^{-u^p}\rd u R(\rd x) \\
&&\quad+e^{-\lambda t} \int_{\mathbb R^d} x \int_0^{|x|^{-1}}\left(e^{-u^p}-e^{-u^p e^{\lambda t p}}\right) u^{-1}\rd u R(\rd x).
\end{eqnarray*}
The proof is similar to that of Theorem \ref{thrm: main}, but with additional care.
\end{remark}

\section{Simulation of TSOU-Processes}\label{sec: sim TS}

Theorem \ref{thrm: main} gives a simple recipe for simulating an increment from a TSOU-process. Its main ingredients are the ability to simulate from a Poisson distribution, an incomplete gamma distribution, the $\mathrm{TS}^p_{\alpha-np}(R_n,0)$ distributions, and distribution $R^1$. Approaches for simulating from a Poisson distribution are well known and an accept-reject algorithm for simulating from the incomplete gamma distribution is given in Section \ref{sec: IGa} above. To simulate from $\mathrm{TS}^p_{\alpha-np}(R_n,0)$ we can use the inverse transform method. Alternatively there are shot noise representations given in \cite{Rosinski:2007} and \cite{Rosinski:Sinclair:2010}. When $n=\gamma-1$ we have $\alpha-np<p$ and we can use the rejection sampling technique developed in \cite{Grabchak:2019}. Approaches for simulating from $R^1$ cannot be easily described since $R^1$ can be, essentially, any probability measure on $\mathbb R^d$. In Section \ref{sec: sim study} we will give a useful example, where simulation from $R^1$ is straightforward. On the other hand, when simulation from $R^1$ is complicated, we can use a modification of the approach given in \cite{Grabchak:2020} for the case $\alpha<p$. We now extend that approach to the case where we allow for any $\alpha$, including $\alpha\ge p$.

The idea is that, sometimes, instead of simulating from $R^1$, it is easier to simulate directly from the distribution of the product $V_jW_j$, where $V_j\sim R^1$ and $W_j\sim\IGa(\alpha,\gamma,p,e^{p\lambda t})$. To do this, it is often easier to work with the family of probability measures $\bar Q$ instead of the Rosi\'nski measure $R$. We begin by defining, for $n=0,1,2,\dots$,
$$
\ell_n(\xi,u) := \frac{1}{n!}(e^{p\lambda t}-1)^n\int_{(0,\infty)} e^{-u^ps} s^n Q_\xi(\rd s), \ \ \xi\in\mathbb S^{d-1}, u>0.
$$
Note that $\ell_0(\xi,u)=q(\xi,u^p)$, where $q(\xi,u)$ is as in \eqref{eq: q to p temp}. Next note that the distribution of the product $V_nW_n$ satisfies, for $B\in\mathfrak B(\mathbb R^d)$,
\begin{eqnarray*}
H(B) &=&  \int_{\mathbb R^d}\int_{0}^\infty 1_{B}(ux)f_{\alpha,\gamma,p,e^{p\lambda t}}(u) \rd u R^1(\rd x)\\
&=&\frac{1}{K} \int_{\mathbb S^{d-1}}\int_0^\infty 1_{B}(u\xi) \left( \ell_0(\xi,u) -  \sum_{n=0}^{\gamma-1}\ell_n(\xi,ue^{\lambda t})u^{np} \right)  u^{-1-\alpha} \rd u  \sigma(\rd \xi),
\end{eqnarray*}
where, for simplicity, we write 
$$
K= K_{\alpha,\gamma,p,e^{p\lambda t}}R(\mathbb R^d) = K_{\alpha,\gamma,p,e^{p\lambda t}}\int_{\mathbb R^d}\int_{(0,\infty)}s^{\alpha/p}Q_\xi(\rd s)\sigma(\rd \xi).
$$
Next, we introduce, the quantities
$$
\kappa_\xi = \frac{1}{ \int_0^\infty\left( \ell_0(\xi,u) -  \sum_{n=0}^{\gamma-1}\ell_n(\xi,ue^{\lambda t})u^{np} \right)  u^{-1-\alpha} \rd u}, \ \ \xi\in\mathbb S^{d-1},
$$
the Borel measure on $\mathbb S^{d-1}$
$$
\sigma_1(\rd \xi) = \frac{1}{\kappa_\xi K}\sigma(\rd \xi),
$$
and the family of Borel measures on $(0,\infty)$ 
$$
F_\xi(\rd u) = f_\xi(u) \rd u, \ \ \xi\in\mathbb S^{d-1},
$$
where
$$
f_\xi(u)= \kappa_\xi \left(  \ell_0(\xi,u) -  \sum_{n=0}^{\gamma-1}\ell_n(\xi,ue^{\lambda t})u^{np}  \right)  u^{-1-\alpha}  , \ \ u>0.
$$
It is not difficult to check that $\sigma_1$ is a probability measure on $\mathbb S^{d-1}$, that $F_\xi$ is a probability measure on $(0,\infty)$ for each $\xi\in\mathbb S^{d-1}$, and that
$$
H(\rd \xi,\rd u)= F_\xi(\rd u)  \sigma_1(\rd \xi), \ \ u>0,\ \xi\in\mathbb S^{d-1}.
$$
Thus, we can simulate $X$ from $H$ by first simulating $\xi$ from $\sigma_1$, then simulating $X_\xi$ from $ F_\xi$, and finally taking  
$$
X=\xi X_\xi.
$$
Note that, even though $X_\xi$ is real-valued, the fact that $\xi\in\mathbb S^{d-1}$ insures that $X$ is $\mathbb R^d$-valued. We now have $X\eqd V_jW_j$. It remains to describe approaches for simulating from $\sigma_1$ and $F_\xi$.
 
Simulation from $\sigma_1$ is straightforward when  $\sigma_1$ is a finite measure, as in this case the problem reduces to simulating from a multinomial distribution. This always holds in the important case where the dimension $d=1$. For the simulation of other distributions on the unit sphere, see the monograph \cite{Johnson:1987}. In particular, there has been much work  focused on the case of a uniform distribution, see, e.g.\ \cite{Tashiro:1977} and the references therein. While no method works in general, one can often set up an approximate simulation method by first approximating $\sigma_1$ by a distribution with a finite support, see Lemma 1 in \cite{Byczkowski:Nolan:Rajput:1993}.

We now turn to the problem of simulation from $F_\xi$ for a fixed $\xi\in\mathbb S^{d-1}$. Toward this end we introduce the quantity
$$
C_{\xi,\gamma} = \frac{(e^{\lambda t p}-1)^\gamma}{\gamma!} \int_{(0,\infty)} s^\gamma Q_\xi(\rd s).
$$
For the remainder of this section we assume that this quantity is finite. By Lemma 7.1 in \cite{Grabchak:2020}, $C_{\xi,\gamma}$ is finite for $\sigma$-a.e.\ $\xi$ if and only if 
\begin{eqnarray*}
\int_{\mathbb R^d} |x|^{\alpha-\gamma p} R(\rd x) <\infty.
\end{eqnarray*}
We next introduce a distribution with pdf
\begin{eqnarray*}
g(u) &=&\alpha(1-\alpha/p) \left(u^{p-\alpha-1}1_{[0<u\le 1]} + u^{-1-\alpha}1_{[u>1]}\right)\\
&=& \frac{\alpha}{p}(p-\alpha)u^{p-\alpha-1}1_{[0<u\le1]} + \left(1-\frac{\alpha}{p}\right)\alpha  u^{-1-\alpha}1_{[u>1]},
\end{eqnarray*}
where $p>\alpha>0$ are parameters. This a type of log-Laplace distribution, see e.g.\ \cite{Kozubowski:Podgorski:2003} and the references therein. It is a mixture of a beta distribution and a Pareto distribution and we will denote it by $\LL(\alpha,p)$. It is readily checked that, if $U_1,U_2\iid U(0,1)$ and 
$$
Y = U_1^{1/(p-\alpha)}U_2^{-1/\alpha},
$$
then $Y\sim\LL(\alpha,p)$. Alternatively, we can use just one random variable $U_1\sim U(0,1)$ and take
$$
Y = \left(\frac{U_1}{\alpha/p}\right)^{1/(p-\alpha)}1_{[U_1\le\alpha/p]} + \left(\frac{1-U_1}{1-\alpha/p}\right)^{-1/\alpha}1_{[U_1>\alpha/p]}. 
$$

\begin{prop}\label{prop:for AR}
We have
$$
f_\xi(u)\le V_2 g_2(u), \ \ u>0,
$$
where $g_2$ is the pdf of the $\LL(\alpha,\gamma p)$ distribution,
$$
V_2 = \kappa_\xi \frac{\gamma p}{\alpha(\gamma p-\alpha)} V_2', 
$$
and 
$$
V_2' = \max\left\{
\min\left\{1,  e^{-\gamma}\gamma^\gamma\frac{(e^{p\lambda t}-1)^{\gamma}}{\gamma!} \right\},C_{\xi,\gamma}
\right\}.
$$
\end{prop}

Let $\varphi_2(\xi,u) = f_\xi(u)/(V_2g_2(u))$ and note that
\begin{eqnarray*}
\varphi_2(\xi,u) &=& \frac{ \ell_0(\xi,u) -  \sum_{n=0}^{\gamma-1}\ell_n(\xi,ue^{\lambda t})u^{np} }{\left(u^{\gamma p}1_{0\le u\le1}+ 1_{u>1}\right) V_2'}\\
&=&\frac{\int_{(0,\infty)}\left(e^{-u^ps} - e^{-su^pe^{p\lambda t}}\sum_{n=0}^{\gamma-1}\frac{(e^{p\lambda t}-1)^n}{n!} s^nu^{np}\right)Q_\xi(\rd s) }{\left(u^{\gamma p}1_{[0\le u\le1]}+ 1_{[u>1]}\right) V_2'}.
\end{eqnarray*}
With this notation we get the following accept-reject algorithm for simulating from $F_\xi$ for a fixed $\xi$.\\

\noindent \textbf{Algorithm 2.}\\
\textbf{Step 1.} Independently simulate $U\sim U(0,1)$ and $Y\sim \LL(\alpha,\gamma p)$.\\
\textbf{Step 2.} If $U\le \varphi_2(Y)$ return $Y$, otherwise go back to step 1.\\

On a given iteration of Algorithm 2, the probability of acceptance is $1/V_2$. We are most interested in the case when $t$ is small. To better understand the behavior of $V_2$ for such $t$ we first note that, by Lemma \ref{lemma: exp approx} given in Section \ref{sec: proofs} below,
\begin{eqnarray*}
\kappa_\xi &\le& \frac{\gamma!}{(e^{p\lambda t}-1)^\gamma \int_0^\infty \int_{(0,\infty)} e^{-u^ps e^{\lambda t p}}s^\gamma Q_\xi(\rd s)u^{\gamma p-\alpha-1}\rd u}\\
&=& \frac{\gamma! p}{(e^{p\lambda t}-1)^\gamma \Gamma(\gamma-\alpha/p)e^{-\lambda t(p\gamma-\alpha)} \int_{(0,\infty)} s^{\alpha/p} Q_\xi(\rd s)}.
\end{eqnarray*}
From here it follows that
$$
\limsup_{t\to0} V_2 \le  \frac{\max\left\{e^{-\gamma}\gamma^\gamma,\int_{(0,\infty)} s^\gamma Q_\xi(\rd s)\right\}}{\Gamma(\gamma-\alpha/p)\int_{(0,\infty)} s^{\alpha/p} Q_\xi(\rd s)}\frac{\gamma p^2}{\alpha(\gamma p-\alpha)} .
$$
Thus the probability of acceptance is bounded away from $0$ when $t$ is small.

In some cases we can improve on Algorithm 2. An issue with the log-Laplace distribution is that it has heavy tails, which can lead to many rejections when the tails of $f_\xi$ are lighter. When the support of $Q_\xi$ is lower bounded, we can replace the log-Laplace distribution with a generalized gamma distribution, which has lighter tails. The method is based on the following result.

\begin{prop}\label{prop: alt bound}
Let $\zeta= \sup\{c>0: Q_\xi((0,c))=0\}$. If $\zeta>0$, then
$$
f_\xi(u)\le V_3 g_3(u), \ \ u>0,
$$
where $g_3$ is the pdf of the $\GGa(p\gamma-\alpha,p,\zeta)$ distribution and
\begin{eqnarray*}
V_3 &=& \kappa_\xi  \zeta^{\alpha/p-\gamma} \frac{\Gamma(\gamma-\alpha/p)}{p} C_{\xi,\gamma}.
\end{eqnarray*}
\end{prop}

Let $\varphi_3(u) = f_\xi(u^{1/p})/(V_3g_3(u^{1/p}))$ and note that
\begin{eqnarray*}
\varphi_3(u) &=& \frac{ \ell_0(\xi,u^{1/p}) -  \sum_{n=0}^{\gamma-1}\ell_n(\xi,u^{1/p}e^{\lambda t})u^{n} }{u^{\gamma}e^{-u\zeta} C_{\xi,\gamma}} \\
&=& \frac{ \int_{(0,\infty)}\left(e^{-us} - e^{-sue^{p\lambda t}}\sum_{n=0}^{\gamma-1}\frac{(e^{p\lambda t}-1)^n}{n!} s^nu^{n}\right)Q_\xi(\rd s) }{u^{\gamma}e^{-u\zeta} C_{\xi,\gamma}} .
\end{eqnarray*}
Combining this with \eqref{eq: sim GGa} leads to the following accept-reject algorithm for simulating from $F_\xi$ for a fixed $\xi$. \\

\noindent \textbf{Algorithm 3.}\\
\textbf{Step 1.} Independently simulate $U\sim U(0,1)$ and $Y\sim \Ga(\gamma-\alpha/p,\zeta)$.\\
\textbf{Step 2.} If $U\le \varphi_3(Y)$ return $Y^{1/p}$, otherwise go back to step 1.\\

It is not difficult to check that
$$
\frac{V_2}{V_3}\ge \frac{\gamma p}{\alpha\Gamma(\gamma-\alpha/p+1)} \zeta^{\gamma-\alpha/p}.
$$
It follows that, $V_2>V_3$ whenever $\zeta>\left(\frac{\alpha}{\gamma p} \Gamma(\gamma-\alpha/p+1)\right)^{1/(\gamma-\alpha/p)}$. In this case, Algorithm 3 will accept with a higher probability than Algorithm 2.\\

\noindent\textbf{Example.} A version of Algorithm 3 was derived in \cite{Kawai:Masuda:2012} for the case of CTS limiting distributions with $\alpha\in(1,2)$. Here $p=1$ and $R(\rd x)= a\zeta^{\alpha}\delta_{1/\zeta}(\rd x)$ for some $a,\zeta>0$. This corresponds to $\sigma(\rd \xi) =a \delta_{1}(\rd\xi)$ and
$$
q_1(u) = e^{-\zeta u} = \int_{(0,\infty)} e^{-us}Q_1(\rd s),
$$
where $Q_1(\rd r) = \delta_\zeta(\rd s)$. It follows that $\gamma=2$ and  
\begin{eqnarray*}
\varphi_3(u) = \frac{2}{(e^{\lambda t}-1)^2 \zeta^2}\frac{e^{-u\zeta}-e^{-ue^{\lambda t}\zeta}-e^{-ue^{\lambda t}\zeta}\zeta u(e^{\lambda t}-1)}{u^2 e^{-u\zeta}}.
\end{eqnarray*}
In this case, our Algorithm 3 reduces to Algorithm 2 in \cite{Kawai:Masuda:2012}. We note that there appears to be a typo in that paper. The formula for what they call $v_{2,\Delta}$ should be as given by $\varphi_3$.

\section{Simulation Study}\label{sec: sim study}

In this section we perform a small-scale simulation study to see how well our methodology works in practice. We focus on a family of one-dimensional $p$-tempered $\alpha$-stable distributions for which the transition law had not been previously derived in the case $\alpha\ge p$. This is the family of power tempered stable distributions, which correspond to the case where $p=1$,
$$
R(\rd x) = .5 c(\alpha+\ell)(\alpha+\ell+1)(1+|x|)^{-2-\alpha-\ell} \rd x,
$$
and $c,\ell>0$ are parameters. When $R$ is of this form, we denote the distribution $\mathrm{TS}^1_\alpha(R,0)$ by $\mathrm{PT}_\alpha(\ell,c)$. These models have a finite mean, but still fairly heavy tails. In fact, if $Y\sim\mathrm{PT}_\alpha(\ell,c)$, then, for $\beta\ge0$,
$$
\rE|Y|^\beta<\infty \mbox{ if and only if } \beta<1+\alpha+\ell.
$$
Thus, $\ell$ controls how heavy the tails of the distribution are. 

Power tempered stable distributions were introduced in \cite{Grabchak:2016book} and then further studied in \cite{Grabchak:2019} and \cite{Grabchak:2020}. However, we use a sightly different parametrization because the one considered in \cite{Grabchak:2019} and \cite{Grabchak:2020} is not continuous at $\alpha=1$. Methods to numerically evaluate the pdfs and related quantities of these distributions are available in the SymTS package \cite{Grabchak:Cao:2017} for the statistical software R. This package also allows for the simulation of random variables from this distribution using the inverse transform method.  For $\alpha\in(0,1)$ the transition laws for the corresponding TSOU-processes were studied in \cite{Grabchak:2020}. However, the case with $\alpha\in[1,2)$ has not been studied before. 

We want to simulate a TSOU-process with parameter $\lambda>0$ and limiting distribution $\mathrm{PT}_\alpha(\ell,c)$, with $\alpha\ge1$, on a discrete grid. For simplicity, we assume that the points are evenly spaced and thus that we want to simulate the observations 
$$
Y_0,Y_t,Y_{2t},\dots, Y_{nt}
$$
for some $t>0$. It is readily checked that $R(\mathbb R^d)=c(\alpha+\ell)<\infty$ and thus that we can use Theorem \ref{thrm: main}. When $\alpha\in[1,2)$, we have $\gamma=2$. Note that, by symmetry, $\int_{\mathbb R}x R(\rd x)=0$ and thus that $b_n=0$ for $n=0,1$. It follows that, if we have simulated $Y_{t(k-1)}=y$, then we can take 
$$
Y_{kt}=  e^{-\lambda t}y+ X_0+ e^{-\lambda t}X_1+ \sum_{j=1}^N V_jW_j,
$$
where $X_0\sim \mathrm{PT}_\alpha(\ell,(1-e^{-\alpha\lambda t})c)$, $X_1\sim \mathrm{PT}_{\alpha-1}(\ell+1,(1-e^{-\lambda t})c)$, $V_1,V_2,\dots\iid R^1$, $W_1,W_2,\dots\iid \mathrm{IGa}(\alpha,2,1,e^{\lambda t})$, and $N\sim \mathrm{Poisson}(\psi)$ are independent random variables. Here, by Proposition \ref{prop: K},
\begin{eqnarray*}
\psi &=& (\alpha+\ell)ce^{-\alpha\lambda t} K_{\alpha,2,1,e^{\lambda t}} \\
 &=& \left\{\begin{array}{ll} 
(\alpha+\ell)c\frac{\Gamma(2-\alpha)}{\alpha(\alpha-1)}\left(e^{-\alpha\lambda t}-1 +\alpha\left(1-e^{-\lambda t}\right)\right) & \alpha\in(1,2)\\
(1+\ell)ce^{-\lambda t}\left(e^{\lambda t}-1 - \lambda t \right) & \alpha=1
\end{array}
\right..
\end{eqnarray*}
To simulate from $R^1$, note that
$$
R^{1}(\rd x) = .5(\alpha+\ell+1)(1+|x|)^{-2-\alpha-\ell} \rd x,
$$
which is a variant of the Pareto distribution. It is not difficult to check that if $U_1\sim U(-1,1)$ and \begin{eqnarray}\label{eq: sim V}
V = \frac{U_1}{|U_1|} \left(|U_1|^{-1/(1+\alpha+\ell)} -1\right),
\end{eqnarray}
then $V\sim R^1$. 

\begin{figure}\label{fig: TSOU}
\linespread{.75}
\begin{tabular}{cc}
\includegraphics[trim={1.15cm 1cm 1cm .5cm},clip,scale=.42]{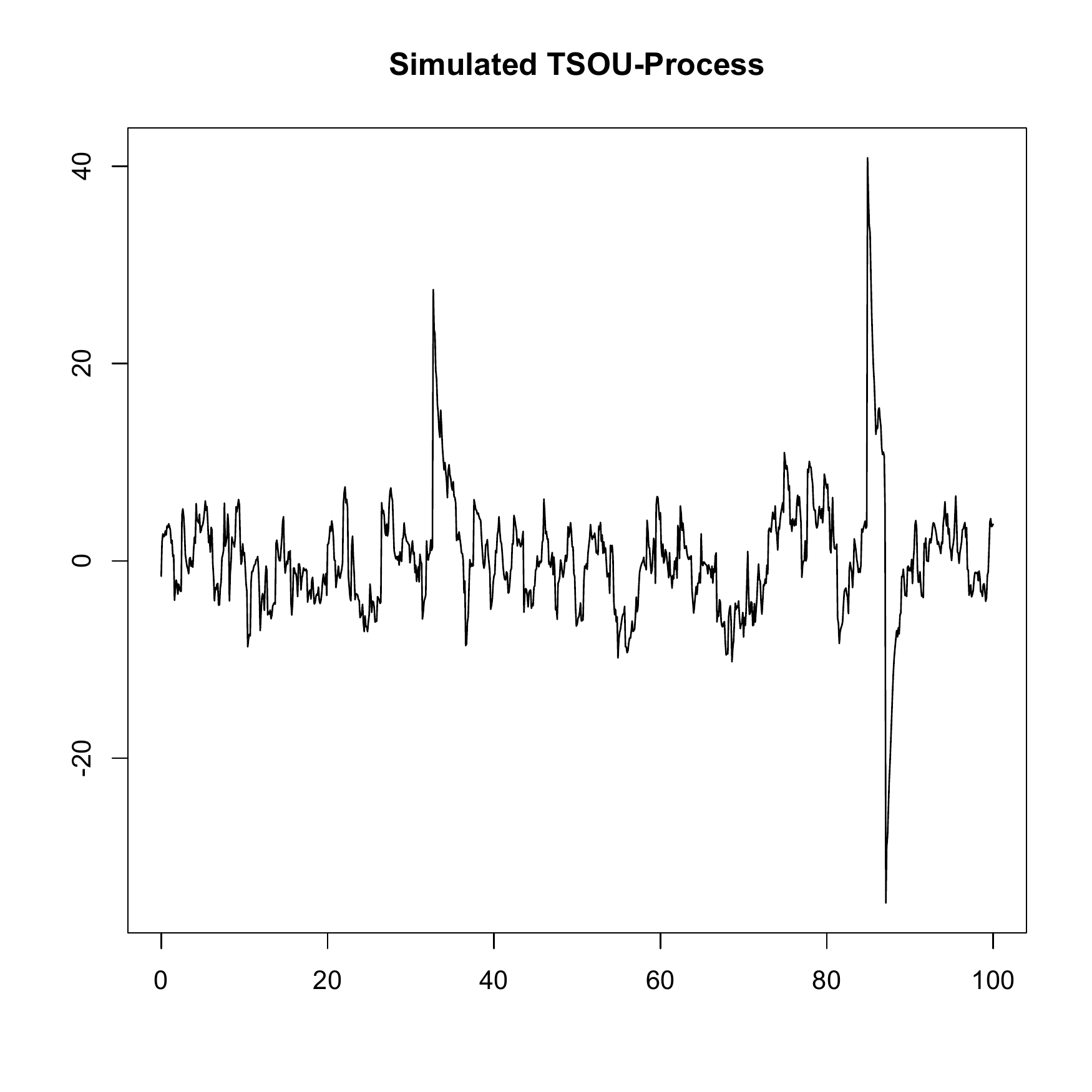} & \includegraphics[trim={1.15cm 1cm 1cm .5cm},clip,scale=.42]{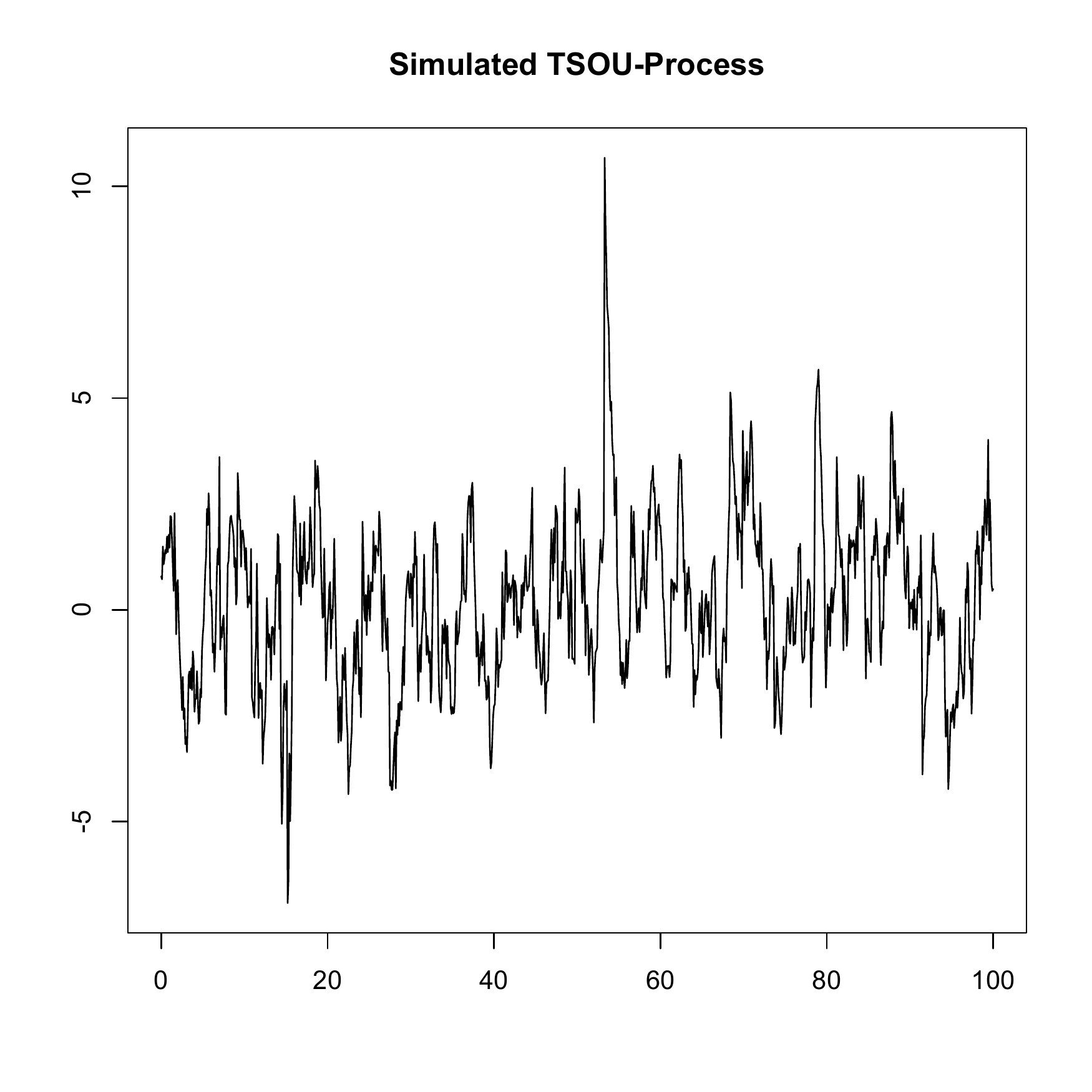} \\
$\alpha=1,\ell=1$ & $\alpha=1,\ell=5$ \vspace{.4cm}\\
\includegraphics[trim={1.15cm 1cm 1cm .5cm},clip,scale=.4]{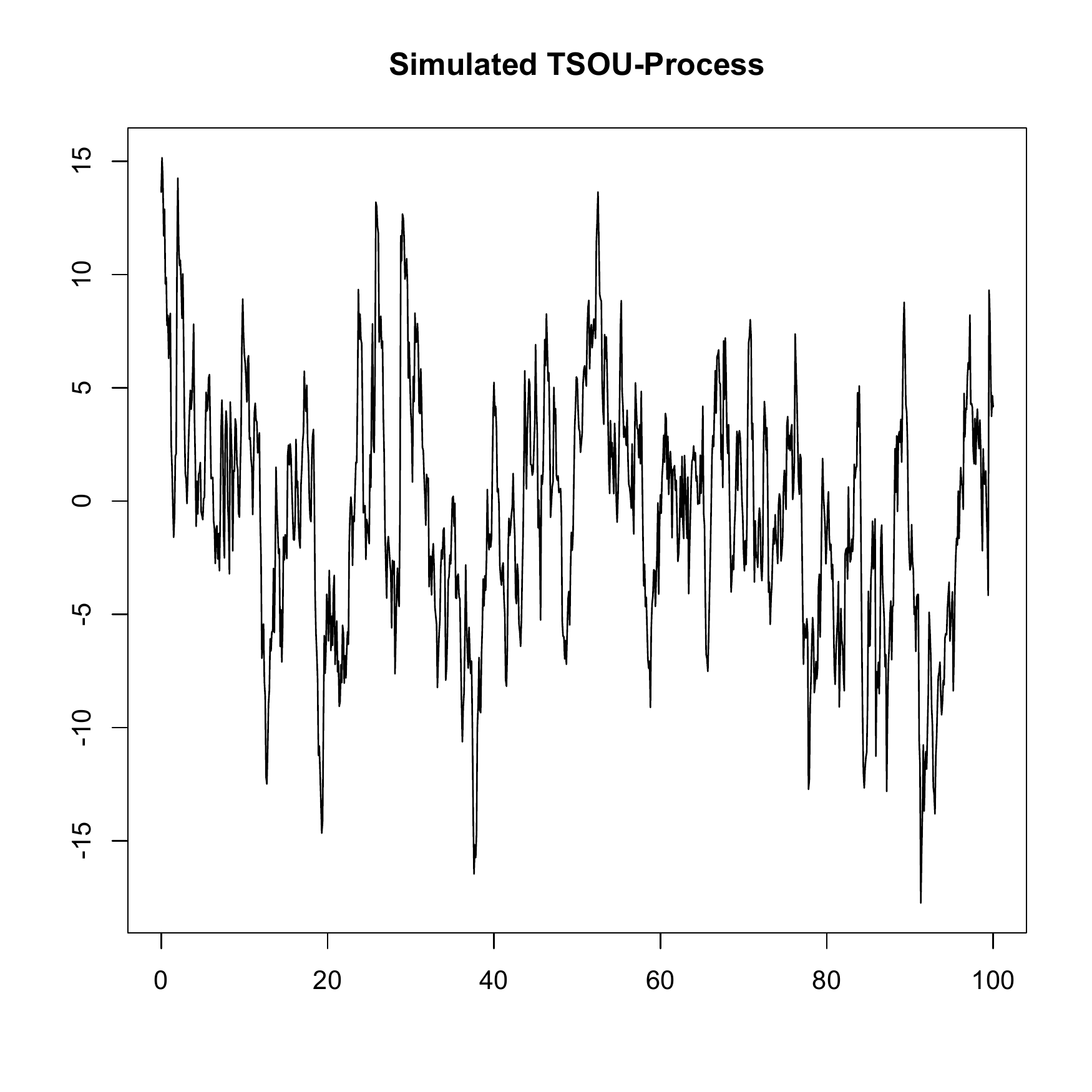} & \includegraphics[trim={1.15cm 1cm 1cm .5cm},clip,scale=.42]{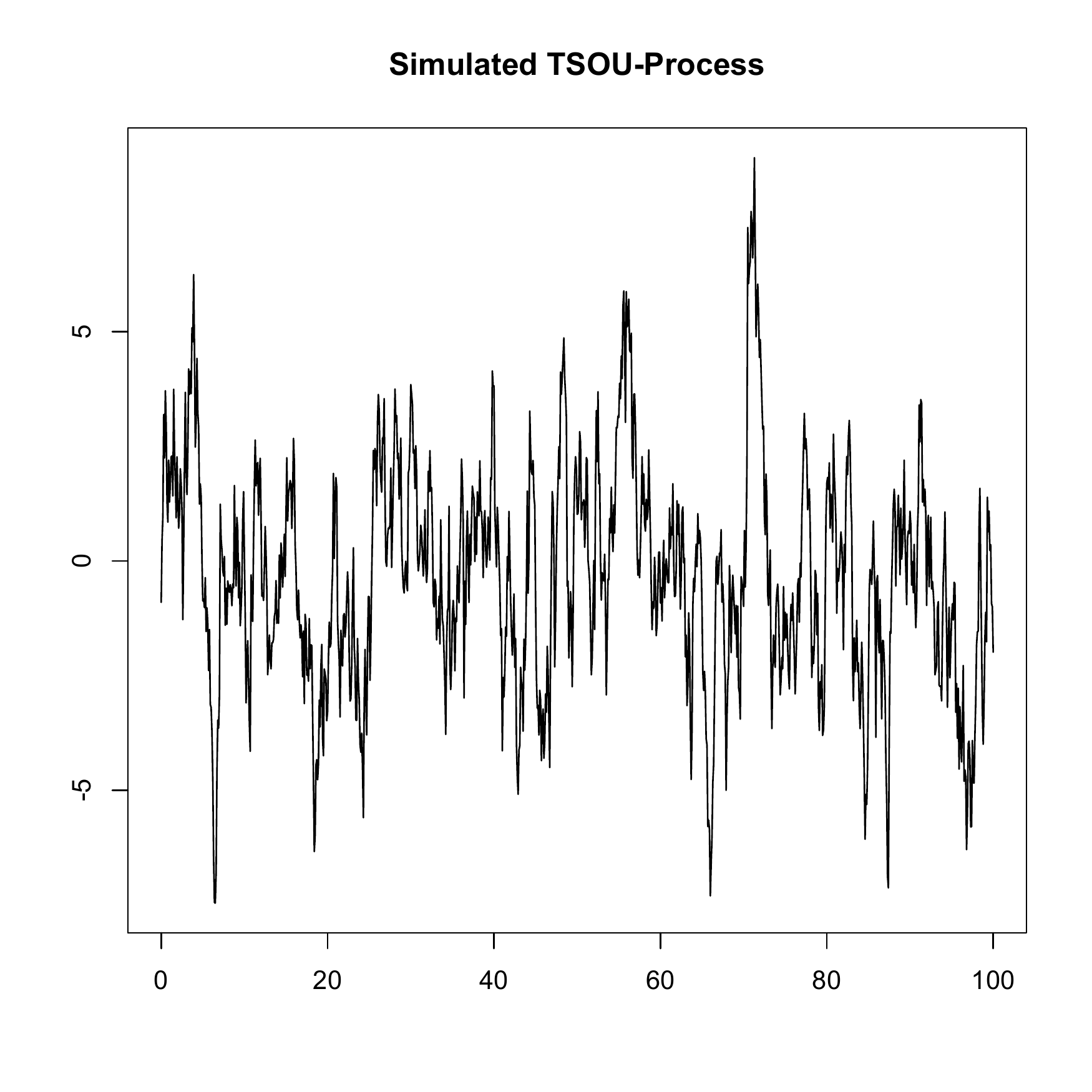} 
\\
$\alpha=1.5,\ell=1$ & $\alpha=1.5,\ell=5$\vspace{.42cm}
\end{tabular}
\vspace{-.5cm}\caption{Simulated TSOU-processes for several choices of the parameters. In all cases the simulated increments are of length $t=0.1$.}
\end{figure}

\begin{figure}\label{fig: limit}
\linespread{.75}
\begin{tabular}{cc}
\includegraphics[trim={1.15cm 1cm 1cm .5cm},clip,scale=.42]{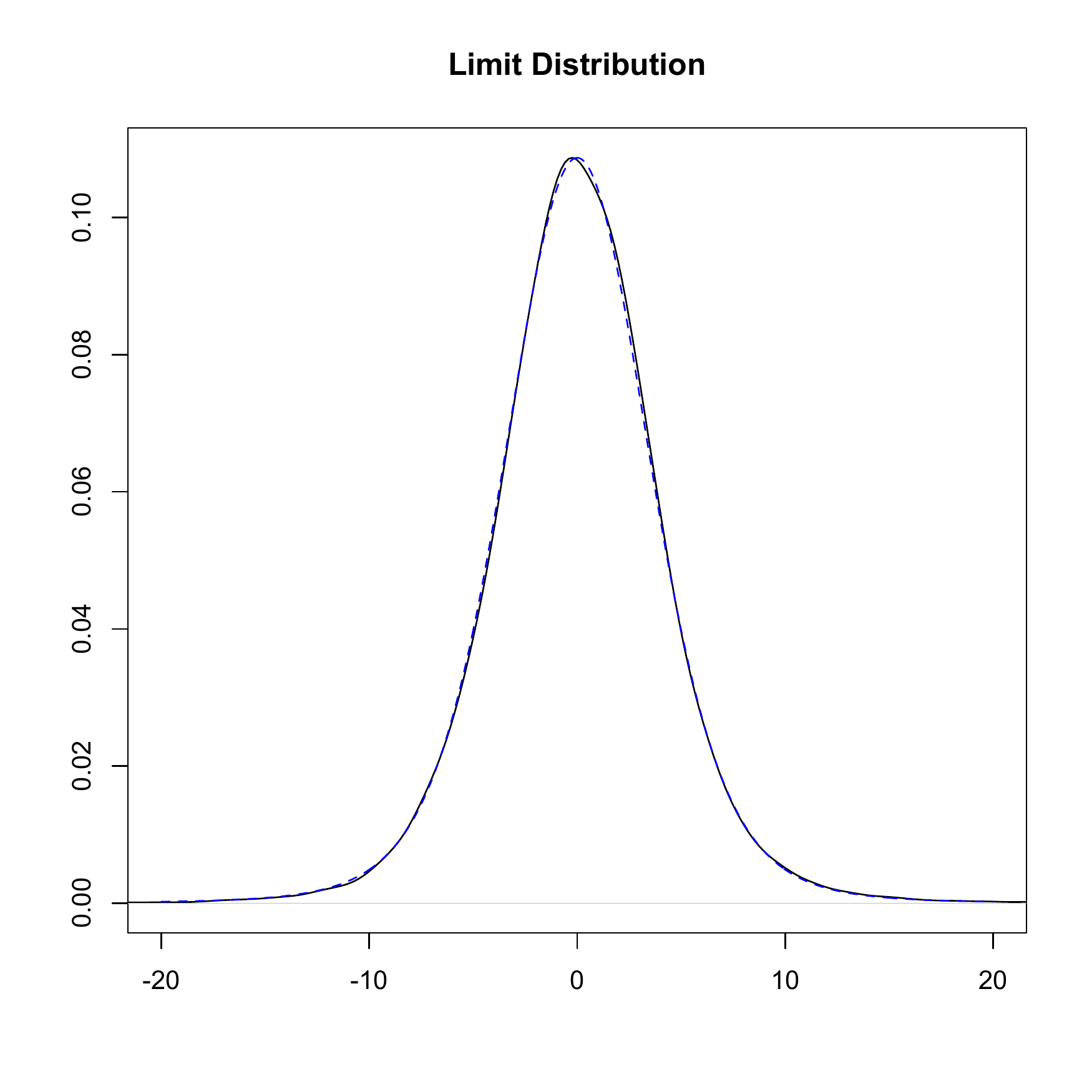} & \includegraphics[trim={1.15cm 1cm 1cm .5cm},clip,scale=.42]{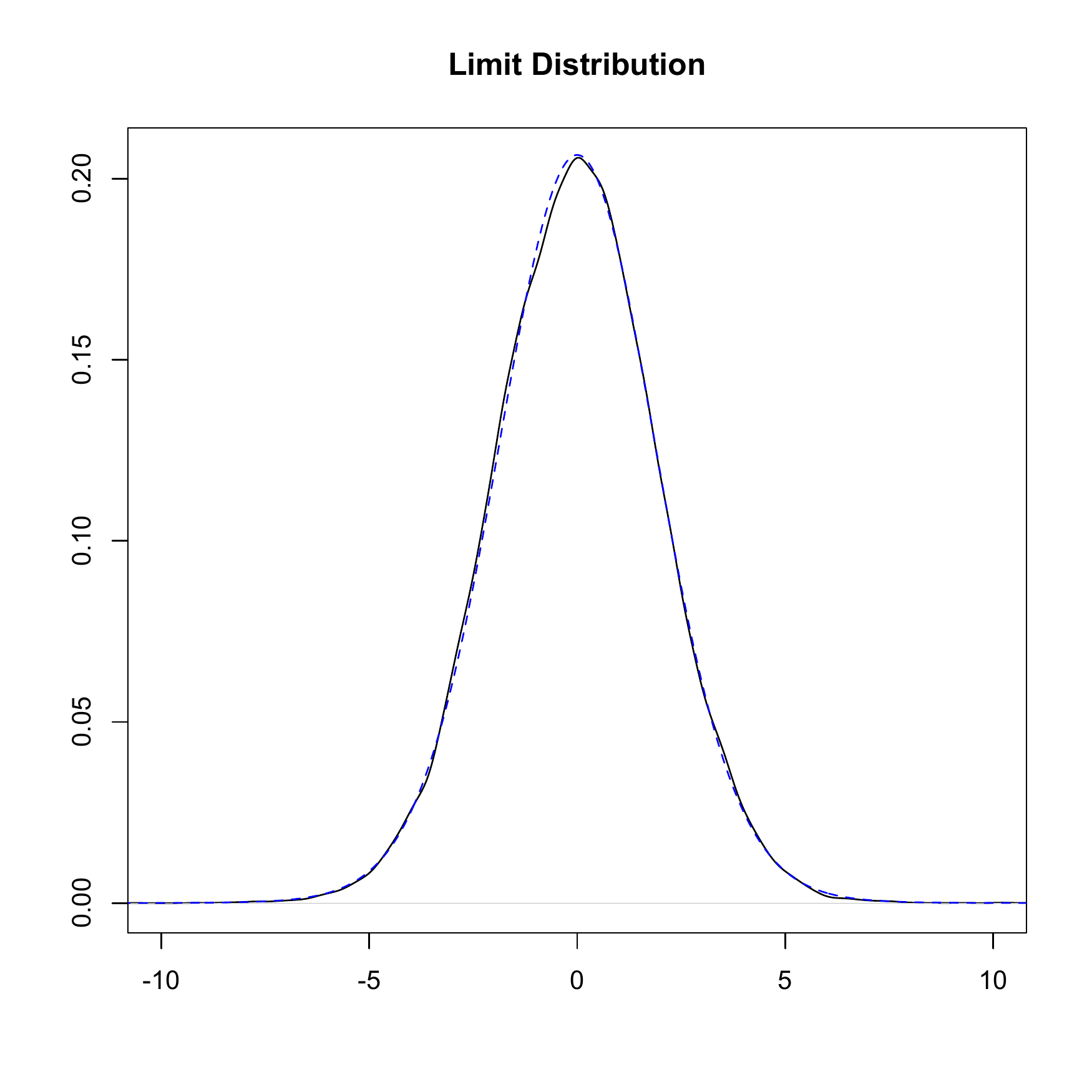} \\
$\alpha=1,\ell=1$ & $\alpha=1,\ell=5$ \vspace{.4cm}\\
\includegraphics[trim={1.15cm 1cm 1cm .5cm},clip,scale=.4]{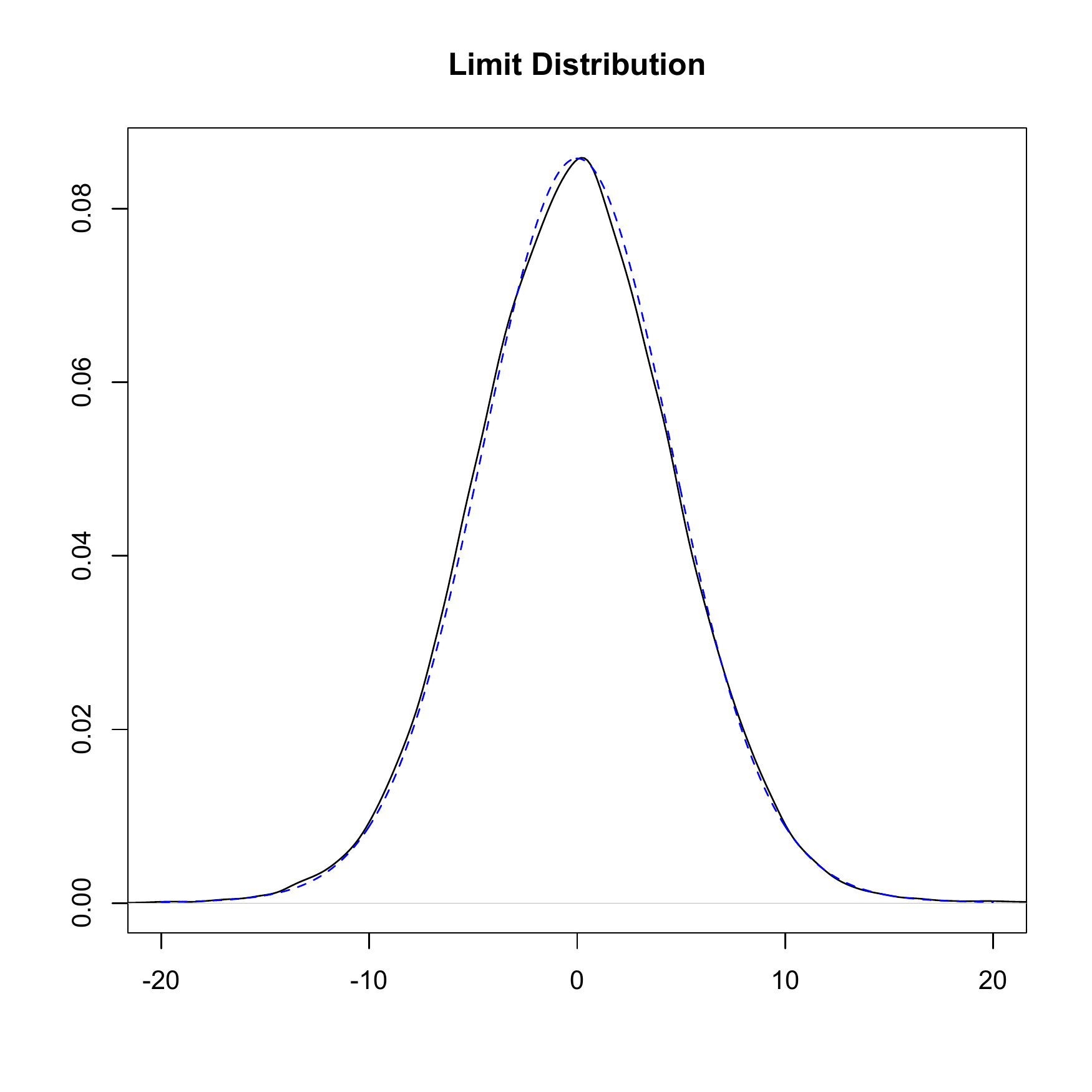} & \includegraphics[trim={1.15cm 1cm 1cm .5cm},clip,scale=.42]{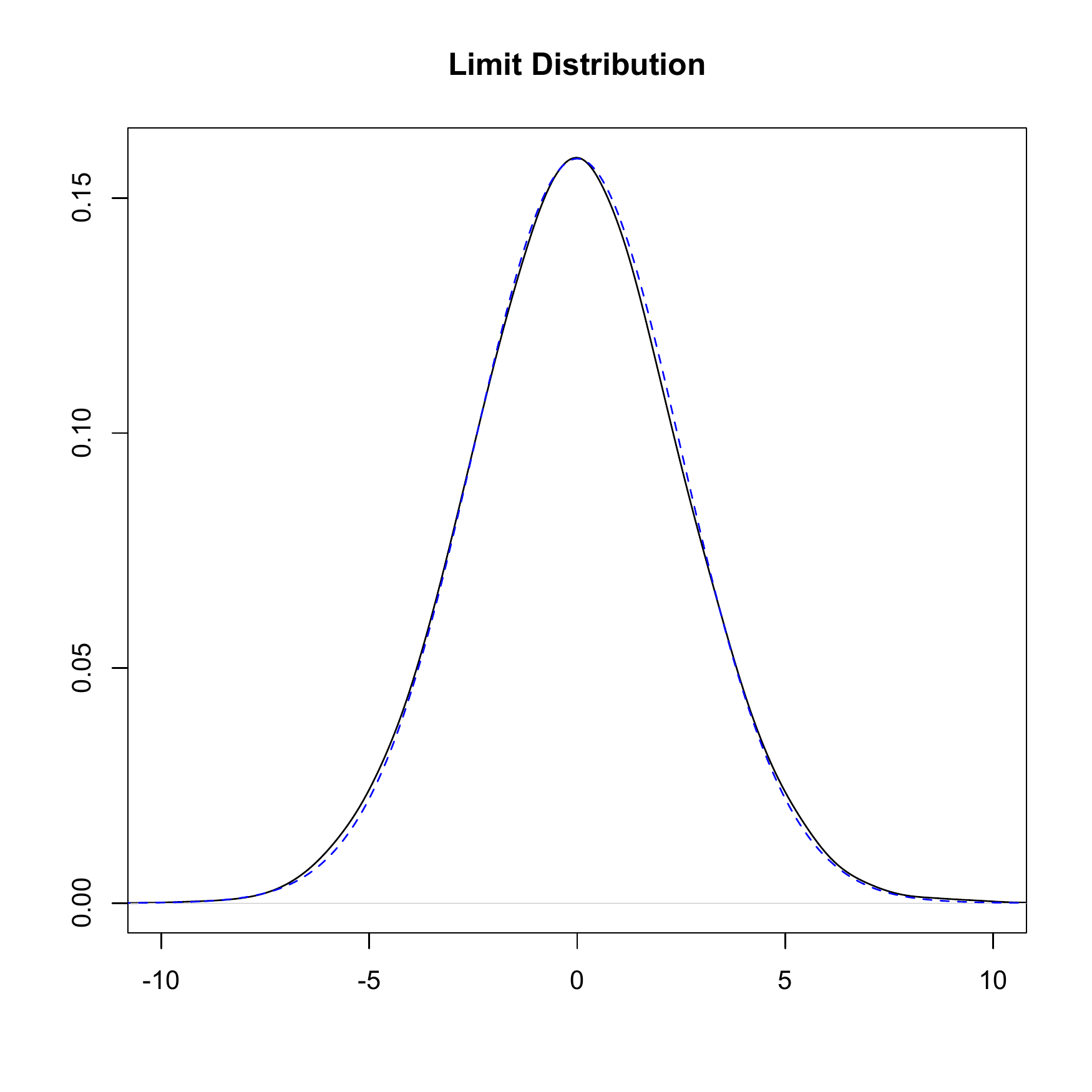}
 \\
$\alpha=1.5,\ell=1$ & $\alpha=1.5,\ell=5$\vspace{.42cm}
\end{tabular}
\vspace{-.5cm}\caption{For each choice of the parameters, we simulate a TSOU-process for $50000$ time steps in increments of $t=0.1$. For each process, we plot the KDE (solid line) overlaid with the true pdf of the limiting distribution (dashed line).}
\end{figure}

For our simulations, we simulate $X_0$ and $X_1$ using the inverse transform method as implemented in the SymTS package, we simulate the $V_i$'s using \eqref{eq: sim V}, and we simulate the $W_i$'s using Algorithm 1. For simplicity, we take $\lambda=1$, $c=10$, and $t=0.1$. We start each path by simulating an observation from the limiting distribution. We then simulate the process at $1000$ time steps. Since the time increment is $t=0.1$, this leads to a simulation of the process up to time $T=100$. Plots of these processes for several choices of $\alpha$ and $\ell$ are given in Figure 1. Further, to check whether we are simulating from the correct limiting distribution, we simulate the process for $50000$ time steps and then plot the kernel density estimator (KDE) based on these observations. Figure 2 gives this plot for each choice of the parameters. The plots are overlaid with the true pdf of the limiting distribution. Since we begin each process  in the limiting distribution, we do not need a burn-in period.

We conclude this section by noting that the result in Theorem \ref{thrm: main} also holds for power tempered stable distributions with $\alpha\in[0,1)$ and can be used for simulation in this case. When $\alpha\in(0,1)$ a different methodology for simulating such TSOU-processes was given in \cite{Grabchak:2020}. However, that approach does not use the fact that, in this case, it is easy to simulate from $R^1$. Instead, it uses a more complicated methodology based on a version of our Algorithm 2. For this reason, we recommend using the methodology suggested by the current paper in this case.

\section{Proofs}\label{sec: proofs}

We begin with a technical lemma.

\begin{lemma}\label{lemma: exp approx}
1. For any $t>0$  and any integer $k\ge1$
$$
\left( 1- e^{-t}\sum_{n=0}^{k-1} \frac{t^n}{n!} \right) =  \frac{1}{(k-1)!}\int_0^{t} e^{-x}x^{k-1}\rd x.
$$
2. For any $0\le a<b$ and any integer $k\ge1$
\begin{eqnarray*}
e^{-a} - e^{-b}\sum_{n=0}^{k-1} \frac{(b-a)^n}{n!} = \frac{e^{-a}}{(k-1)!}\int_0^{b-a} e^{-x}x^{k-1}\rd x
\end{eqnarray*}
and
$$
\frac{e^{-b}}{k!}(b-a)^{k}\le e^{-a} - e^{-b}\sum_{n=0}^{k-1} \frac{(b-a)^n}{n!} \le \frac{e^{-a}}{k!}(b-a)^{k}. 
$$
\end{lemma}

\begin{proof}
The first part follows by integration by parts and induction on $k$. The second follows immediately from the first.
\end{proof}

\begin{proof}[Proof of Proposition \ref{prop: K}.]
First note that
\begin{eqnarray*}
K_{\beta,\gamma,p,\eta} &=& \frac{1}{p\Gamma(\gamma)}\int_0^\infty \int_{0}^{u(\eta-1)}e^{-x-u} x^{\gamma-1} \rd x u^{-1-\beta/p}\rd u \\ 
&=& \frac{1}{p\Gamma(\gamma)}\int_0^\infty \int_{u}^{u\eta}e^{-v} (v-u)^{\gamma-1} \rd v u^{-1-\beta/p}\rd u\\
&=& \frac{1}{p\Gamma(\gamma)}\int_0^\infty e^{-v} \int^{v}_{v/\eta}(v-u)^{\gamma-1}u^{-1-\beta/p} \rd u \rd v\\
&=& \frac{1}{p\Gamma(\gamma)}\int_0^\infty e^{-v}v^{\gamma-1} \int^{v}_{v/\eta}(1-u/v)^{\gamma-1}u^{-1-\beta/p} \rd u \rd v\\
&=& \frac{\Gamma(\gamma-\beta/p)}{p\Gamma(\gamma)}\int^{1}_{1/\eta}(1-u)^{\gamma-1}u^{-1-\beta/p} \rd u,
\end{eqnarray*}
where the first, second, and fifth lines follow by change of variables. The asymptotic formulas follow by L'H\^ opital's rule, except in the case when $\eta\to\infty$ and $\beta<0$. In this case, they follow by basic properties of the beta function. The last part of the proposition follows by applying the Binomial Theorem.
\end{proof}

\begin{proof}[Proof of Proposition \ref{prop: IGa bound}.]
Note that
\begin{eqnarray*}
f_{\beta,\gamma,p,\eta}(u) &\le& \frac{1}{\Gamma(\gamma)K_{\beta,\gamma,p,\eta}}  \int_0^{u^p(\eta-1)} x^{\gamma-1}\rd x e^{-u^p} u^{-1-\beta}\\
&=& \frac{(\eta-1)^\gamma}{\Gamma(\gamma+1)K_{\beta,\gamma,p,\eta}} e^{-u^p} u^{p\gamma-\beta-1}.
\end{eqnarray*}
From here the result is immediate.
\end{proof}

\begin{proof}[Proof of Proposition \ref{prop: BDLP for TS}.]
Theorem 2.17 in \cite{Rocha-Arteaga:Sato:2019} gives the formula for the shift (after taking into account the fact that we are using a different parametrization) and implies that
\begin{eqnarray*}
M(B) &=& \int_{\mathbb S^{d-1}} \int_0^\infty 1_{B}\left(u \xi\right)\left(-\frac{\partial}{\partial u}(q(\xi,u^p)u^{-\alpha}) \right)  \rd u \sigma(\rd \xi)\\
&=& \int_{\mathbb S^{d-1}} \int_{(0,\infty)} \int_0^\infty 1_{B}\left(u \xi\right)\left(spu^p+\alpha \right)u^{-1-\alpha} e^{-u^ps} \rd u Q_\xi(\rd s) \sigma(\rd \xi)\\
&=& \int_{\mathbb R^{d}} \int_0^\infty 1_{B}\left(u \frac{x}{|x|}\right)\left(|x|pu^p+\alpha \right)u^{-1-\alpha} e^{-u^p|x|} \rd u Q(\rd x)\\
&=& \int_{\mathbb R^{d}} \int_0^\infty 1_{B}\left(u\frac{x}{|x|^{1+1/p}}\right)\left(pu^p+\alpha \right)u^{-1-\alpha} e^{-u^p} \rd u |x|^{\alpha/p}Q(\rd x).
\end{eqnarray*}
Now applying \eqref{eq: R} gives the result.
\end{proof}

\begin{lemma}\label{lemma:second int}
In the context of Theorem \ref{thrm: main}, $Y$ is a Markov process with temporally homogenous transition function $P_t(y,\rd x)$ 
having characteristic function
$\int_{\mathbb R^d}e^{i\langle x,z\rangle} P_t(y,\rd x) = \exp\left\{C_t(y,z)\right\}$, where
\begin{eqnarray*}
C_t(y,z) &=& e^{-\alpha\lambda t}R(\mathbb R^d)K_{\alpha,\gamma,p,e^{\lambda tp}} \int_{\mathbb R^{d}}\int_0^\infty \psi_{0}(z,ux) f_{\alpha,\gamma,p,e^{\lambda tp}}(u) \rd u  R^1(\rd x)\\
&&\quad+ \int_{\mathbb R^d}\int_0^\infty \psi_\alpha(z,ux) u^{-1-\alpha} e^{-u^p}\rd u R_0(\rd x)  \\
&&\quad + \sum_{n=1}^{\gamma-1} \int_{\mathbb R^{d}}  \int_0^\infty \psi_{\alpha-np}(ze^{-\lambda t},ux)e^{-u^p}  u^{-1-(\alpha-np)} \rd u R_n(\rd x)\\
&& \quad+ ie^{-\lambda t}\langle y,z\rangle +  i \left(1-e^{-\lambda t}\right)\langle  b,z\rangle - \sum_{n=0}^{\gamma-1} i\langle z,b_n\rangle
\end{eqnarray*}
and
\begin{eqnarray*}
\psi_\alpha(z,x) = e^{i\langle z,x \rangle} - 1 - i\langle z,x \rangle 1_{[\alpha\ge1]}.
\end{eqnarray*}
\end{lemma}

\begin{proof}
Proposition 2.13 in \cite{Rocha-Arteaga:Sato:2019} implies that
\begin{eqnarray*}
C_t(y,z) =  ie^{-\lambda t}\langle y,z\rangle +  i   \left(1-e^{-\lambda t}\right) \langle  b,z\rangle +  \lambda \int_0^t  \int_{\mathbb R^d}\psi_\alpha(e^{-\lambda s}z,x) M(\rd x)\rd s,
\end{eqnarray*}
where $M$ is as in Proposition \ref{prop: BDLP for TS}. Now using the fact that $\psi_\alpha(az,x)=\psi_\alpha(z,ax)$ for any $a\in\mathbb R$
\begin{eqnarray*}
&&\lambda\int_0^t  \int_{\mathbb R^d}\psi_\alpha(e^{-\lambda s}z,x) M(\rd x)\rd s \\
&&\qquad = \lambda\int_{\mathbb R^d}\int_0^\infty \int_0^t  \psi_\alpha(z,xue^{-\lambda s})(\alpha+pu^p)u^{-1-\alpha} e^{-u^p}\rd s \rd uR(\rd x)\\
&&\qquad = \int_{\mathbb R^d}\int_0^\infty \int_{ue^{-\lambda t}}^u  \psi_\alpha(z,xv)(\alpha+pu^p)u^{-1-\alpha} e^{-u^p}v^{-1}\rd v \rd uR(\rd x)\\
&&\qquad = \int_{\mathbb R^d}\int_0^\infty  \psi_\alpha(z,xv)\int_{v}^{ve^{\lambda t}}  (\alpha+pu^p)u^{-1-\alpha} e^{-u^p} \rd u v^{-1}\rd v R(\rd x)\\
&&\qquad = \int_{\mathbb R^d}\int_0^\infty  \psi_\alpha(z,xv)\left(e^{-v^p}-e^{-\alpha\lambda t}e^{-v^pe^{p\lambda t}}\right) v^{-1-\alpha}\rd v R(\rd x)\\
&&\qquad= \left(1-e^{-\lambda t\alpha}\right)\int_{\mathbb R^d}\int_0^\infty \psi_\alpha(z,ux) u^{-1-\alpha} e^{-u^p}\rd u R(\rd x)  \\
&&\qquad\quad + e^{-\alpha\lambda t} \int_{\mathbb R^{d}}\int_0^\infty \psi_\alpha(z,ux) \left( e^{-u^p}-e^{-u^p e^{\lambda tp}} \right)  u^{-1-\alpha} \rd u  R(\rd x),
\end{eqnarray*}
where the fifth line follows by the fact that 
\begin{eqnarray}\label{eq: deriv}
-\frac{\rd}{\rd u}u^{-\alpha}e^{-u^p} =  (\alpha+pu^p)u^{-1-\alpha} e^{-u^p}.
\end{eqnarray}
From here we just need to put the last line into the appropriate form. This line can be written as 
\begin{eqnarray*}
&&\sum_{n=1}^{\gamma-1} e^{(pn-\alpha)\lambda t} \int_{\mathbb R^{d}}  \int_0^\infty \psi_{\alpha}(z,ux)e^{-u^p e^{\lambda tp}}  u^{-1-(\alpha-np)} \rd u R_n(\rd x)\\
&& \quad+ e^{-\alpha\lambda t} \int_{\mathbb R^{d}}\int_0^\infty \psi_\alpha(z,ux) \left(e^{-u^p} - e^{-u^p e^{\lambda tp}}\sum_{n=0}^{\gamma-1}  \frac{(e^{p\lambda t}-1)^n}{n!} u^{np} \right)  u^{-1-\alpha} \rd u  R(\rd x)\\
&& = \sum_{n=1}^{\gamma-1} \int_{\mathbb R^{d}}  \int_0^\infty \psi_{\alpha}(z,e^{-\lambda t}ux)e^{-u^p}  u^{-1-(\alpha-np)} \rd u R_n(\rd x) \\
&& \quad+ e^{-\alpha\lambda t} K_{\alpha,\gamma,p,e^{\lambda tp}} \int_{\mathbb R^{d}}\int_0^\infty \psi_{\alpha}(z,ux) f_{\alpha,\gamma,p,e^{\lambda tp}}(u) \rd u  R(\rd x)\\
&& = \sum_{n=1}^{\gamma-1} \int_{\mathbb R^{d}}  \int_0^\infty \psi_{\alpha-np}(e^{-\lambda t}z,ux)e^{-u^p}  u^{-1-(\alpha-np)} \rd u R_n(\rd x) - \sum_{n=0}^{\gamma-1}i\langle z,b_n\rangle\\
&& \quad+ e^{-\alpha\lambda t}R(\mathbb R^d)K_{\alpha,\gamma,p,e^{\lambda tp}} \int_{\mathbb R^{d}}\int_0^\infty \psi_{0}(z,ux) f_{\alpha,\gamma,p,e^{\lambda tp}}(u) \rd u  R^1(\rd x),
\end{eqnarray*}
which completes the proof.
\end{proof}

\begin{proof}[Proof of Theorem \ref{thrm: main}.]
The result follows by noting that that the characteristic function of the random variable given by \eqref{eq: trans RV} is $\int_{\mathbb R^d}e^{i\langle x,y\rangle} P_t(y,\rd x) = \exp\left\{C_t(y,z)\right\}$, where $C_t(y,z)$ is of the form required by Lemma  \ref{lemma:second int}.
\end{proof}

\begin{proof}[Proof of Proposition \ref{prop:for AR}.]
Lemma \ref{lemma: exp approx} implies that
\begin{eqnarray*}
f_\xi(u) &=&\kappa_\xi \int_{(0,\infty)} \left( e^{-u^ps} -e^{-u^p e^{p\lambda t}s}\sum_{n=0}^{\gamma-1}\frac{1}{n!}(e^{p\lambda t}-1)^ns^nu^{np} \right) Q_\xi(\rd s)u^{-\alpha-1}\\
&\le&  \kappa_\xi\frac{(e^{p\lambda t}-1)^{\gamma}}{\gamma!}u^{-\alpha-1} \int_{(0,\infty)} e^{-u^p s} (su^p)^\gamma Q_\xi(\rd s).\end{eqnarray*}
It follows that
$$
f_\xi(u)\le  \kappa_\xi\frac{(e^{p\lambda t}-1)^{\gamma}}{\gamma!}u^{\gamma p-\alpha-1} \int_{(0,\infty)} e^{-u^p s} s^\gamma Q_\xi(\rd s)\le \kappa_\xi C_{\xi,\gamma}u^{p\gamma-\alpha-1}.
$$
and similarly, since $e^{-x}x^\gamma \le e^{-\gamma}\gamma^\gamma$ for $x\ge0$ and $Q_\xi$ is a probability measure, we have
$$
f_\xi(u)\le e^{-\gamma}\gamma^\gamma \frac{\kappa_\xi}{\gamma!}(e^{p\lambda t}-1)^{\gamma}u^{-\alpha-1}.
$$
On the other hand, since $0\le \ell_0(\xi,u)\le 1$ and $\sum_{n=0}^{\gamma-1}\ell_n(\xi,ue^{\lambda t})u^{np}\ge0$ it follows that for any $u>0$
\begin{eqnarray*}
f_\xi(u) = \kappa_\xi\left(\ell_0(\xi,u) -\sum_{n=0}^{\gamma-1}\ell_n(\xi,ue^{\lambda t})u^{np}\right)u^{-1-\alpha} \le \kappa_\xi u^{-1-\alpha}.
\end{eqnarray*}
Combining these three bounds gives the result.
\end{proof}

\begin{proof}[Proof of Proposition \ref{prop: alt bound}.]
Lemma \ref{lemma: exp approx} implies that
\begin{eqnarray*}
f_\xi(u) &=& \kappa_\xi \int_{[\zeta,\infty)} \left( e^{-u^ps} -e^{-u^p e^{p\lambda t}s}\sum_{n=0}^{\gamma-1}\frac{(e^{p\lambda t}-1)^n}{n!}s^nu^{np} \right) Q_\xi(\rd s)u^{-1-\alpha}\\
&\le&\kappa_\xi  \frac{(e^{p\lambda t}-1)^\gamma}{\gamma!} \int_{[\zeta,\infty)} e^{-u^ps} s^\gamma Q_\xi(\rd s) u^{p\gamma-\alpha-1}\\
&\le&\kappa_\xi  C_{\xi,\gamma} e^{-u^p \zeta}  u^{p\gamma-1-\alpha}.
\end{eqnarray*}
From here the result follows.
\end{proof}

\end{document}